\documentclass[12pt,tbtags,leqno]{amsart}

\usepackage{amssymb}
\usepackage{amsthm}
\usepackage{amsmath}
\usepackage{verbatim}
\usepackage[numbers,sort&compress]{natbib}
\usepackage{CJK}

\usepackage{geometry}
\numberwithin{equation}{section}

\newcommand{\HM}{\mathcal{M}}
\newcommand{\HE}{\mathcal{E}}

\newcommand{\HN}{\mathcal{N}}

\newcommand{\D}{\mathbb{D}}

\newcommand{\C}{\mathbb{C}}

\newcommand{\T}{\mathbb{T}}

\newcommand{\ran}{\mathrm{ran \ }}
\newcommand{\ind}{\mathrm{ind}}

\newcommand{\rank}{\mathrm{rank \ }}

\theoremstyle{plain}
\newtheorem{theorem}{Theorem}[section]
\newtheorem{lemma}[theorem]{Lemma}

\newtheorem{prop}[theorem]{Proposition}
\newtheorem{corollary}[theorem]{Corollary}

\theoremstyle{definition}

\begin{document}
\title{Hilbert-Schmidtness of some finitely generated submodules in $H^2(\D^2)$}

\author[S. Luo]{Shuaibing Luo$^1$}
\address{S. Luo: College of Mathematics and Econometrics, Hunan University, Changsha, Hunan, 410082, PR China}
\email{shuailuo2@126.com}

\author[K. J. Izuchi]{Kei Ji Izuchi$^2$}
\address{K. J. Izuchi: Department of Mathematics,
Niigata University, Niigata 950-2181, Japan}
\email{izuchi@m.sc.niigata-u.ac.jp}

\author[R. Yang]{Rongwei Yang$^3$}
\address{R. Yang: School of Mathematical Science, Tianjin Normal University, Tianjin, PR China \& Department of Mathematics and Statistics, SUNY at Albany, Albany, NY 12222, USA}
\email{ryang@albany.edu}

\noindent \thanks{$^1$ the first author, supported by the NSF of China (11701167, 11771132). $^2$ supported by JSPS KAKENHI Grant (15K04895). $^3$ corresponding author, supported by Tianjin Thousand Talents Plan (ZX0471601033).}

\maketitle
\begin{center}
To the memory of Takahiko Nakazi
\end{center}

\begin{abstract}
A closed subspace $\HM$ of the Hardy space $H^2(\D^2)$ over the bidisk is called a submodule if it is invariant under multiplication by coordinate functions $z_1$ and $z_2$. Whether every finitely generated submodule is Hilbert-Schmidt is an unsolved problem. This paper proves that every finitely generated submodule $\HM$ containing $z_1 - \varphi(z_2)$ is Hilbert-Schmidt, where $\varphi$ is any finite Blaschke product.
Some other related topics such as fringe operator and Fredholm index are also discussed.
\medskip

\end{abstract}


\section{Introduction}
Let $H^2(\D^2)$ be the Hardy space over the bidisk $\D^2$. If we denote the variables by $z_1$ and $z_2$, then $H^2(\D^2)$ can be identified with $H^2(z_1)\otimes H^2(z_2)$, where $H^2(z)$ is the Hardy space over the unit disk $\D$ with variable denoted by $z$. Let $M_{z_1}$ and $M_{z_2}$ be the multiplication operators with symbols $z_1$ and $z_2$, respectively. A closed subspace $\HM$ of $H^2(\D^2)$ is called a submodule if $\HM$ is invariant under $M_{z_1}$ and $M_{z_2}$. It is easy to see that a submodule is indeed a module over the polynomial ring $C[z_1,z_2]$ with module action defined by multiplication of functions. We denote the lattice of submodules by $Lat(H^2(\D^2))$. Beurling's theorem fully characterizes the submodule of the classical Hardy space $H^2(\D)$. It says that any submodule of $H^2(\D)$ is of the form $\theta H^2(\D)$ for some inner function $\theta$. If we denote by $R_z$ and $S_z$ the restriction of $M_z$ on $\HM$ and respectively the compression of $M_z$ on $\HM^\perp = H^2(\D) \ominus \HM$, then it is not hard to check that $R_z$ and $S_z$ are Fredholm operators, and their indices are $-1$ and $0$, respectively. However, submodules of $H^2(\D^2)$ are complicated (\cite{Ru69}) and they bear no similar characterization. The research on $H^2(\D^2)$ is ongoing. One approach to this problem is to study some relatively simple submodules, and hope that the study will generate concepts and techniques for the general picture. In analogy with the operators $R_z$ and $S_z$ on $H^2(\D)$, we are interested in the operator pairs $(R_{z_1},R_{z_2})$ and $(S_{z_1},S_{z_2})$ on $H^2(\D^2)$. It is clear that $(R_{z_1},R_{z_2})$ is a pair of commuting isometries, and $(S_{z_1},S_{z_2})$ is a pair of commuting contractions. These pairs contain much information about $\HM$ and they are the subjects of many recent studies.

Suppose $\HM$ is a submodule of $H^2(\D^2)$, i.e. $\HM \in Lat(H^2(\D^2))$. Let
$$C = I - R_{z_1}R_{z_1}^* - R_{z_2}R_{z_2}^* + R_{z_1}R_{z_2}R_{z_1}^*R_{z_2}^*.$$
$C$ is called the core operator or defect operator for $\HM$ (\cite{GY04}). $\HM$ is called a Hilbert-Schmidt submodule if the core operator $C$ is Hilbert-Schmidt. Hilbert-Schmidt submodules have many good properties and have been studied extensively in the literature, see e.g. \cite{III17, Ya99, Ya01, Ya04, Ya05} and the references therein. In particular, it was shown in \cite{Ya05} that $C^2$ is unitarily equivalent to
$$\left(
\begin{matrix}
[R_{z_1}^*,R_{z_1}][R_{z_2}^*,R_{z_2}][R_{z_1}^*,R_{z_1}]&0\\
0&[R_{z_1}^*,R_{z_2}][R_{z_2}^*,R_{z_1}]
\end{matrix}
\right).
$$
This implies that $C$ is Hilbert-Schmidt (or compact) if and only if $[R_{z_1}^*,R_{z_1}][R_{z_2}^*,R_{z_2}]$ and $[R_{z_1}^*,R_{z_2}]$ are both Hilbert-Schmidt (or compact). It is known that if $C$ is Hilbert-Schmidt, then the pairs $(R_{z_1},R_{z_2})$ and $(S_{z_1},S_{z_2})$ are Fredholm. Almost all known examples of submodules are Hilbert-Schmidt. The only known non-Hilbert-Schmidt submodule is the submodule $\HM$ with $\dim \HM\ominus (z_1\HM + z_2\HM) = \infty$, in which case $[R_{z_1}^*,R_{z_1}][R_{z_2}^*,R_{z_2}]$ is not compact (\cite{Ya01}). Further, if $\HM$ is Hilbert-Schmidt then it can be shown that $z_1\HM + z_2\HM$ is closed. It is not clear whether this is true for all submodules $\HM$. For $\lambda \in \D^2$, let \[\ind_\lambda \HM = \dim \HM \ominus ((z_1-\lambda_1)\HM + (z_2-\lambda_2)\HM).\] The integer $\ind_\lambda \HM$ is called the index of $\HM$ at $\lambda$. It captures important information of $\HM$ and was studied in \cite{LR}. It is not hard to see that $\ind_\lambda \HM$ is less than or equal to the rank of $\HM$, so if there exists a sequence of $\lambda_n \in \D^2$ such that $\ind_{\lambda_n}\HM$ goes to infinity, then $\HM$ is not finitely generated. It is conjectured in \cite{Ya99} that every finitely generated submodule is Hilbert-Schmidt. This paper confirms the conjecture for submodules containing function $z_1 - \varphi(z_2)$, where $\varphi$ is a finite Blaschke product.

In 2008, the second and the third author studied the submodules $\HM$ generated by $z_1 - \varphi(z_2)$, where $\varphi$ is an inner function (\cite{IY08}), and showed that $\HM = [z_1 - \varphi(z_2)]$ is Hilbert-Schmidt. Moreover, the quotient module $H^2(\D^2) \ominus [z_1 - \varphi(z_2)]$ can be identified with $(H^2(z_2) \ominus \varphi H^2(z_2)) \otimes L^2_a(\D)$ and $S_{z_1}$ is unitarily equivalent to $I \otimes M_z$ on $(H^2(z_2) \ominus \varphi H^2(z_2)) \otimes L^2_a(\D)$, where $L^2_a(\D)$ is the Bergman space. When $\varphi(z_2) = z_2$, this recovers the well-known fact that $S_{z_i} (i = 1, 2)$ on $H^2(\D^2) \ominus [z_1 - z_2]$ is unitarily equivalent to the Bergman shift. In this paper, we look at submodules $\HM$ which contain $z_1 - \varphi(z_2)$, where $\varphi$ is a finite Blaschke product. We obtain a necessary and sufficient condition for such $\HM$ to be Hilbert-Schmidt. As an application, submodules which contain $z_1 - z_2$ are fully characterized. The main result of the paper is the following theorem.
\begin{theorem}\label{hlsmnmcdv}
Let $\varphi$ be a finite Blaschke product and $\HM \in Lat(H^2(\D^2))$ contain $z_1 - \varphi(z_2)$. Then $\HM$ is a Hilbert-Schmidt submodule if and only if $\HM$ is finitely generated.
\end{theorem}
In Section 2, we define and study the fringe operator $F_\lambda$, where $\lambda \in \D^2$, and show that $F_\lambda$ is Fredholm if and only if the pair $(R_{z_1}-\lambda_1,R_{z_2}-\lambda_2)$ is Fredholm. This result will be used in the proof of Theorem \ref{hlsmnmcdv2} and Proposition \ref{dmfass}. In Section 3, we prove Theorem \ref{hlsmnmcdv}. When submodules contain $z_1 - z_2$, we also determine the dimensions of the cohomology vector spaces for the pairs $(R_{z_1}-\lambda_1,R_{z_2}-\lambda_2)$ and $(S_{z_1}-\lambda_1,S_{z_2}-\lambda_2), \lambda \in \D^2$ (see Proposition \ref{dmfass}).

\section{Fringe operator}
Suppose $\HM$ is a submodule of $H^2(\D^2)$. For $\lambda \in \D^2$, we define the fringe operator $F_\lambda$ on $\HM \ominus (z_1 - \lambda_1) \HM$ by
$$F_\lambda f = P_{\lambda_1} M_{z_2 - \lambda_2} f,\quad f \in \HM \ominus (z_1 - \lambda_1) \HM,$$
where $P_{ \lambda_1}$ is the orthogonal projection from $\HM$ to $\HM \ominus (z_1 - \lambda_1) \HM$. The fringe operator was introduced and studied by the third author in \cite{Ya01}, where the fringe operator $F_{(0,0)}$ was mainly investigated. Let $\varphi_{\lambda_i}(z) = \varphi_{\lambda_i} (z_i) = \frac{z_i - \lambda_i}{1 - \overline{\lambda_i}z_i}, i =1, 2$, and define $\widetilde{F_\lambda} f = P_{\lambda_1} M_{\varphi_{\lambda_2}} f$ for $f \in \ran P_{\lambda_1}$. Then one verifies that $\ran \widetilde{F_\lambda} = \ran F_\lambda$ and $\ker \widetilde{F_\lambda} = \ker F_\lambda$. Let $R_{\varphi_{\lambda_i}} = M_{\varphi_{\lambda_i}}|\HM$ and $P_\HE$ stand for the orthogonal projection from $H^2(\D^2)$ to the closed subspace $\HE$. The following lemma and proposition generalize corresponding facts in \cite{Ya01}.
\begin{lemma}
$\ran F_\lambda = [(z_1 - \lambda_1)\HM + (z_2 - \lambda_2)\HM]\ominus (z_1 - \lambda_1) \HM$.
\end{lemma}
\begin{proof}
If $g \in \HM \ominus (z_1 - \lambda_1) \HM$, then
$$F_\lambda g = (z_2 - \lambda_2)g - (P_\HM - P_{\lambda_1}) (z_2 - \lambda_2)g \in (z_1 - \lambda_1)\HM + (z_2 - \lambda_2)\HM.$$
Conversely, let $h = (z_1 - \lambda_1) f + (z_2 - \lambda_2) g \in [(z_1 - \lambda_1)\HM + (z_2 - \lambda_2)\HM]\ominus (z_1 - \lambda_1) \HM$. If $g \in (z_1 - \lambda_1)\HM$, then it is clear that $h = 0$ and $F_\lambda 0 = 0$. So suppose $g \in \HM \ominus (z_1 - \lambda_1) \HM$. Note that
\begin{align*}
h &= (z_1 - \lambda_1) f + (z_2 - \lambda_2) g\\
& = (z_1 - \lambda_1) f + P_{\lambda_1} (z_2 - \lambda_2) g + (P_\HM - P_{\lambda_1}) (z_2 - \lambda_2) g,
\end{align*}
and $(z_1 - \lambda_1) f + (P_\HM - P_{\lambda_1}) (z_2 - \lambda_2) g \in (z_1 - \lambda_1) \HM$. This implies
$$(z_1 - \lambda_1) f + (P_\HM - P_{\lambda_1}) (z_2 - \lambda_2) g = 0.$$
It then follows that
$$F_\lambda g = (z_2 - \lambda_2)g - (P_\HM - P_{\lambda_1}) (z_2 - \lambda_2)g = (z_2 - \lambda_2)g + (z_1 - \lambda_1) f = h.$$
The proof is complete.
\end{proof}
It follows from the above lemma that $\ker F_\lambda^* = \HM \ominus [(z_1 - \lambda_1)\HM + (z_2 - \lambda_2)\HM]$ and $\dim \ker F_\lambda^* = \ind_\lambda \HM$. The following two propositions will be used in the proof of Proposition \ref{indeve}. Let $P_{ \lambda_2}$ be the orthogonal projection from $\HM$ to $\HM \ominus (z_2 - \lambda_2) \HM$. For convenience, we let $p = P_\HM$.
\begin{prop}\label{rlspoaci}
For $f \in \HM \ominus (z_1 - \lambda_1) \HM$, we have\\
(i) $f - \widetilde{F_\lambda}^*\widetilde{F_\lambda} f= [R_{\varphi_{\lambda_2}}^*, R_{\varphi_{\lambda_1}}] [R_{\varphi_{\lambda_1}}^*,R_{\varphi_{\lambda_2}}]f$;\\
(ii) $f - \widetilde{F_\lambda}\widetilde{F_\lambda}^* f= [R_{\varphi_{\lambda_1}}^*, R_{\varphi_{\lambda_1}}] [R_{\varphi_{\lambda_2}}^*,R_{\varphi_{\lambda_2}}]f$.
\end{prop}
\begin{proof}
(i) If $f \in (z_1 - \lambda_1)\HM$, then $[R_{\varphi_{\lambda_1}}^*,R_{\varphi_{\lambda_2}}] f = 0$. Thus $[R_{\varphi_{\lambda_1}}^*,R_{\varphi_{\lambda_2}}] = [R_{\varphi_{\lambda_1}}^*,R_{\varphi_{\lambda_2}}] P_{\lambda_1}$. Since $R_{\varphi_{\lambda_1}}^* P_{\lambda_1} = 0$, we have
\begin{align*}
&[R_{\varphi_{\lambda_1}}^*,R_{\varphi_{\lambda_2}}] = [R_{\varphi_{\lambda_1}}^*,R_{\varphi_{\lambda_2}}] P_{\lambda_1}\\
& = R_{\varphi_{\lambda_1}}^*R_{\varphi_{\lambda_2}} P_{\lambda_1}\\
& = p M_{\varphi_{\lambda_1}}^* M_{\varphi_{\lambda_2}}P_{\lambda_1}.
\end{align*}
Hence
\begin{align}\label{dcpipt}
[R_{\varphi_{\lambda_2}}^*, R_{\varphi_{\lambda_1}}] [R_{\varphi_{\lambda_1}}^*,R_{\varphi_{\lambda_2}}] = P_{\lambda_1} M_{\varphi_{\lambda_2}}^* M_{\varphi_{\lambda_1}} p M_{\varphi_{\lambda_1}}^* M_{\varphi_{\lambda_2}}P_{\lambda_1}.
\end{align}
On the other hand, for $f \in \HM \ominus (z_1 - \lambda_1) \HM$,
\begin{align}\label{fmftsft}
& f - \widetilde{F_\lambda}^*\widetilde{F_\lambda}f = f - P_{\lambda_1} M_{\varphi_{\lambda_2}}^* P_{\lambda_1} M_{\varphi_{\lambda_2}} f \notag\\
& = f - [P_{\lambda_1} f - P_{\lambda_1} M_{\varphi_{\lambda_2}}^* (p -P_{\lambda_1}) M_{\varphi_{\lambda_2}} f] \notag\\
& = P_{\lambda_1} M_{\varphi_{\lambda_2}}^* (p -P_{\lambda_1}) M_{\varphi_{\lambda_2}} f \notag\\
& = P_{\lambda_1} M_{\varphi_{\lambda_2}}^* M_{\varphi_{\lambda_1}} p M_{\varphi_{\lambda_1}}^* M_{\varphi_{\lambda_2}} f,
\end{align}
where in the last equality we used $(p -P_{\lambda_1}) M_{\varphi_{\lambda_2}} f = M_{\varphi_{\lambda_1}} p M_{\varphi_{\lambda_1}}^* M_{\varphi_{\lambda_2}} f$. Therefore the conclusion follows from (\ref{dcpipt}) and (\ref{fmftsft}).

(ii) Note that $P_{\lambda_1} M_{\varphi_{\lambda_2}}^*P_{\lambda_1} = p M_{\varphi_{\lambda_2}}^*P_{\lambda_1}$. Hence for $f \in \HM \ominus (z_1 - \lambda_1) \HM$,
\begin{align*}
f - \widetilde{F_\lambda}\widetilde{F_\lambda}^* f& = f - P_{\lambda_1} M_{\varphi_{\lambda_2}}P_{\lambda_1} M_{\varphi_{\lambda_2}}^*f\\
& = P_{\lambda_1}f - P_{\lambda_1} M_{\varphi_{\lambda_2}}p M_{\varphi_{\lambda_2}}^*f\\
& = P_{\lambda_1} [p - pM_{\varphi_{\lambda_2}}p M_{\varphi_{\lambda_2}}^*p]P_{\lambda_1} f\\
& = P_{\lambda_1} P_{\lambda_2} f.
\end{align*}
Since $[R_{\varphi_{\lambda_1}}^*, R_{\varphi_{\lambda_1}}] = P_{\lambda_i}$ are projections onto $\HM \ominus (z_i - \lambda_i)\HM$, the assertion follows from the above equation.
\end{proof}

If $[R_{z_1}^*, R_{z_2}]$ is compact, then $[R_{\varphi_{\lambda_1}}^*,R_{\varphi_{\lambda_2}}]$ is compact for every $\lambda \in \D^2$. Hence in this case the above proposition implies that for every $\lambda \in \D^2$, the fringe operator $F_\lambda = F_{(\lambda_1,0)} - \lambda_2$ is left semi-Fredholm.

Similarly, for $\lambda \in \D^2$ we let $G_\lambda$ and $\widetilde{G_\lambda}$ be defined by $G_\lambda f = P_{\lambda_2} M_{z_1 - \lambda_1} f$ and $\widetilde{G_\lambda} f = P_{\lambda_2} M_{\varphi_{\lambda_1}} f$ for $f \in \ran P_{\lambda_2}$. Then $G_\lambda$ and $\widetilde{G_\lambda}$ have the same range and kernel. The following result is thus parallel to Proposition 2.2.
\begin{prop}\label{srsfgl}
(i) $\ran G_\lambda = [(z_1 - \lambda_1)\HM + (z_2 - \lambda_2)\HM]\ominus (z_2 - \lambda_2) \HM$;\\
(ii) for $f \in \HM \ominus (z_2 - \lambda_2) \HM$, $f - \widetilde{G_\lambda}^*\widetilde{G_\lambda} f= [R_{\varphi_{\lambda_1}}^*,R_{\varphi_{\lambda_2}}] [R_{\varphi_{\lambda_2}}^*, R_{\varphi_{\lambda_1}}] f$;\\
(iii) for $f \in \HM \ominus (z_2 - \lambda_2) \HM$, $f - \widetilde{G_\lambda}\widetilde{G_\lambda}^* f=[R_{\varphi_{\lambda_2}}^*,R_{\varphi_{\lambda_2}}] [R_{\varphi_{\lambda_1}}^*, R_{\varphi_{\lambda_1}}] f$.
\end{prop}

Now we discuss the Koszul complex of the pair $R = (R_{z_1}, R_{z_2})$. The Koszul complex of $R$ is defined by
$$K(R): 0 \xrightarrow{\partial_R^{-1}} \HM \xrightarrow{\partial_R^{0}} \HM \oplus \HM \xrightarrow{\partial_R^{1}} \HM \xrightarrow{\partial_R^{2}} 0,$$
where $\partial_R^{0} f = (R_{z_1} f, R_{z_2}f)$ and $\partial_R^{1} (f, g) = R_{z_1}g - R_{z_2} f$. The pair $R$ is called a Fredholm pair if all the maps have closed range and the cohomology vector space $\ker \partial_R^{i}/\ran \partial_R^{i-1}$ is finite dimensional for $i = 0, 1$ and $2$ (see \cite{Cu81, GRS05}). If $R$ is a Fredholm pair, then the index of $R$ is defined by
$$\ind R = \sum_{i=0}^2 (-1)^i \dim (\ker \partial_R^{i}/\ran \partial_R^{i-1}) = \ind_{(0,0)}\HM - \dim (\ker \partial_R^{1}/\ran \partial_R^{0}).$$
The essential Taylor spectrum of $R$ is defined to be
$$\sigma_e(R) = \{\lambda \in \C^2: R - \lambda ~~\text{is not Fredholm}\}.$$

For $\lambda \in \D^2$, we have
$$\ker \partial_{R - \lambda}^1 = \{(f,g): (z_1 - \lambda_1)g = (z_2 - \lambda_2)f, f, g \in \HM\},$$
$$\ran \partial_{R - \lambda}^0 = \{((z_1 - \lambda_1)f, (z_2 - \lambda_2)f): f \in \HM\}.$$
Let $\ker \widetilde{\partial}^1 = \{(f,g): \varphi_{\lambda_1}g = \varphi_{\lambda_2}f, f, g \in \HM\}$ and $\ran \widetilde{\partial}^0 = \{(\varphi_{\lambda_1}f, \varphi_{\lambda_2}f): f \in \HM\}$. It is easy to see that the map $U: \ker \partial_{R - \lambda}^1 \rightarrow \ker \widetilde{\partial}^1$ defined by
$$U (f,g) = ((1 - \overline{\lambda_2}z_2)f, (1 - \overline{\lambda_1}z_1)g)$$
is one-to-one and onto. Observe that
$$U(\partial_{R - \lambda}^1 / \ran \partial_{R - \lambda}^0) = \ker \widetilde{\partial}^1 / \ran \widetilde{\partial}^0,$$
thus $\ker \partial_{R - \lambda}^1 \ominus \ran \partial_{R - \lambda}^0$ is isomorphic to $\ker \widetilde{\partial}^1 \ominus \ran \widetilde{\partial}^0$. We show in the following that $\ker F_\lambda$ is isomorphic to $\ker \partial_{R - \lambda}^1 \ominus \ran \partial_{R - \lambda}^0$.
\begin{lemma}\label{rlbfadots}
Let $\ker \widetilde{\partial}^1$ and $\ran \widetilde{\partial}^0$ be as above, then
\begin{align*}
\ker \widetilde{\partial}^1 \ominus \ran \widetilde{\partial}^0 &= \{(f,g): g = M_{\varphi_{\lambda_1}}^* M_{\varphi_{\lambda_2}}f, f \in \ran P_{\lambda_1}, M_{\varphi_{\lambda_2}} f \in \varphi_{\lambda_1} \HM\}.\\
& = \{(f,g): f = M_{\varphi_{\lambda_2}}^* M_{\varphi_{\lambda_1}}g, g \in \ran P_{\lambda_2}, M_{\varphi_{\lambda_1}} g \in \varphi_{\lambda_2} \HM\}.
\end{align*}
\end{lemma}
\begin{proof}
We prove the first equality, the other equality follows from a similar argument. We first show that the set $\HN = \{(f,g): g = M_{\varphi_{\lambda_1}}^* M_{\varphi_{\lambda_2}}f, f \in \ran P_{\lambda_1}, M_{\varphi_{\lambda_2}} f \in \varphi_{\lambda_1} \HM\}$ is contained in $\ker \widetilde{\partial}^1 \ominus \ran \widetilde{\partial}^0$. Let $(f,g) \in \HN$. Then $M_{\varphi_{\lambda_2}} f \in \varphi_{\lambda_1} \HM$ and $g = M_{\varphi_{\lambda_1}}^* M_{\varphi_{\lambda_2}}f$. So $\varphi_{\lambda_1} g = \varphi_{\lambda_2}f$, i.e. $(f,g) \in \ker \widetilde{\partial}^1$. Note that $(\varphi_{\lambda_1} h, \varphi_{\lambda_2} h) \in \ran \widetilde{\partial}^0$, and
\begin{align}\label{inpbtfo}
\langle(f,g), (\varphi_{\lambda_1} h, \varphi_{\lambda_2} h)\rangle& = \langle f, \varphi_{\lambda_1} h\rangle + \langle g, \varphi_{\lambda_2} h\rangle\\
& = 2 \langle f, \varphi_{\lambda_1} h\rangle \notag\\
& = 0. \notag
\end{align}
It thus follows that $\HN \subseteq \ker \widetilde{\partial}^1 \ominus \ran \widetilde{\partial}^0$.

Conversely, if $(f,g) \in \ker \widetilde{\partial}^1 \ominus \ran \widetilde{\partial}^0$, then $\varphi_{\lambda_1} g = \varphi_{\lambda_2}f \in \varphi_{\lambda_1} \HM$. So $g = M_{\varphi_{\lambda_1}}^* M_{\varphi_{\lambda_2}}f$. Using (\ref{inpbtfo}) we conclude that $f \in \ran P_{\lambda_1}$. Thus $\ker \widetilde{\partial}^1 \ominus \ran \widetilde{\partial}^0 \subseteq \HN$, and hence they are the same.
\end{proof}

Since $\ker F_\lambda = \ker\widetilde{F_\lambda} = \{f \in \ran P_{\lambda_1}: M_{\varphi_{\lambda_2}} f \in \varphi_{\lambda_1} \HM\}$, the above lemma implies that $\ker F_\lambda$ is isomorphic to $\ker \widetilde{\partial}^1 \ominus \ran \widetilde{\partial}^0$, and hence $\ker F_\lambda$ is isomorphic to $\ker \partial_{R - \lambda}^1 \ominus \ran \partial_{R - \lambda}^0$. Recall that $\ker F_\lambda^* = \HM \ominus [(z_1 - \lambda_1)\HM + (z_2 - \lambda_2)\HM]$. It follows that if $F_\lambda$ is left semi-Fredholm, then
\begin{align}\label{indbfoaf}
\ind F_\lambda &= \dim \ker F_\lambda - \dim \ker F_\lambda^*\notag\\
& = \dim (\ker \partial_{R - \lambda}^1 \ominus \ran \partial_{R - \lambda}^0) - \ind_\lambda \HM.
\end{align}
Thus $F_\lambda$ is Fredholm if and only if $R- \lambda$ is Fredholm, in which case the above equation implies
\begin{align}\label{rlbindofr}
\ind F_\lambda = -\ind (R - \lambda).
\end{align}

Next we look at the Koszul complex of $S = (S_{z_1}, S_{z_2})$, where $S_{z_i} = P_{\HM^\perp}M_{z_i}|\HM^\perp, i = 1, 2$. The Koszul complex of $S$ is defined similarly by
$$K(S): 0 \xrightarrow{\partial_S^{-1}} \HM^\perp \xrightarrow{\partial_S^{0}} \HM^\perp\oplus\HM^\perp \xrightarrow{\partial_S^{1}} \HM^\perp \xrightarrow{\partial_S^{2}} 0.$$
The pair $S$ is a Fredholm pair if the vector spaces $\ker \partial_S^{i}/\ran \partial_S^{i-1}$ are finite dimensional. If $S$ is a Fredholm pair, then the index of $S$ is
\begin{align}\label{indos}
\ind S &= \sum_{i=0}^2 (-1)^i \dim (\ker \partial_S^{i}/\ran \partial_S^{i-1}) \notag\\
& = \dim \ker \partial_S^0 - \dim (\ker \partial_S^{1}/\ran \partial_S^{0}) + \dim (\ran \partial_S^{1})^\perp.
\end{align}
For earlier work on the index of $(S_{z_1}, S_{z_2})$ we refer readers to \cite{Ya06, LYY11} and the references therein. Observe that
\begin{align*}
\ker \partial_{S-\lambda}^{0} &= \{f \in \HM^\perp: S_{z_1 - \lambda_1} f = S_{z_2 - \lambda_2} f = 0\} \\
& = \ker S_{\varphi_{\lambda_1}} \cap \ker S_{\varphi_{\lambda_1}}.
\end{align*}
We show in the following that $\ker F_\lambda$ is isomorphic to $\ker \partial_{S-\lambda}^{0}$.
\begin{lemma}\label{isobkfaks}
$\ker \widetilde{F_\lambda} = M_{\varphi_{\lambda_1}} \ker \partial_{S-\lambda}^{0}$.
\end{lemma}
\begin{proof}
Let $f \in \ker \partial_{S-\lambda}^{0}$. Then $\varphi_{\lambda_i} f \in \HM, i =1, 2$ and $\varphi_{\lambda_1} f \perp \varphi_{\lambda_1} \HM$. Hence $\varphi_{\lambda_1} f \in \ran P_{\lambda_1}$. Note that $\varphi_{\lambda_2} f \in \HM$ implies that $\varphi_{\lambda_2}  \varphi_{\lambda_1} f \in \varphi_{\lambda_1} \HM$. We thus conclude that $\varphi_{\lambda_1} f \in \ker \widetilde{F_\lambda}$, and so $M_{\varphi_{\lambda_1}} \ker \partial_{S-\lambda}^{0} \subseteq \ker \widetilde{F_\lambda}$.

For containment in the other direction, if $f \in \ker \widetilde{F_\lambda}$, then $\varphi_{\lambda_2} f \in \varphi_{\lambda_1} \HM$. This implies $f(\lambda_1, \cdot) = 0$ and $\frac{f}{\varphi_{\lambda_1}} \in \HM^\perp$. From $f \in \ran P_{\lambda_1}$ and $\varphi_{\lambda_2} \frac{f}{\varphi_{\lambda_1}} \in \HM$, we obtain$S_{\varphi_{\lambda_1}}\frac{f}{\varphi_{\lambda_1}} = S_{\varphi_{\lambda_2}}\frac{f}{\varphi_{\lambda_1}} = 0$. Therefore $\frac{f}{\varphi_{\lambda_1}} \in \ker \partial_{S-\lambda}^{0}$, and $\ker \widetilde{F_\lambda} \subseteq M_{\varphi_{\lambda_1}} \ker \partial_{S-\lambda}^{0}$. So $\ker \widetilde{F_\lambda} = M_{\varphi_{\lambda_1}} \ker \partial_{S-\lambda}^{0}$.
\end{proof}
Recall that $\ker F_\lambda$ is isomorphic to $\ker \partial_{R - \lambda}^1 \ominus \ran \partial_{R - \lambda}^0$, we thus obtain the following lemma.
\begin{lemma}\label{isosps}
Let $\HM \in Lat(H^2(\D^2))$. Then the spaces $\ker F_\lambda, \ker \partial_{R - \lambda}^1 \ominus \ran \partial_{R - \lambda}^0$ and $\ker \partial_{S-\lambda}^{0}$ are isomorphic for each $\lambda \in \D^2$.
\end{lemma}

\section{Hilbert-Schmidtness}
In this section, we study the Hilbert-Schmidtness of submodules containing some particular functions and prove our main theorem.
\subsection{Submodules containing $z_1 - \varphi(z_2)$} In this subsection, we consider the submodules which contain $z_1 - \varphi(z_2)$, where $\varphi$ is an inner function. Let $\varphi$ be an inner function and $M_\varphi = [z_1 - \varphi(z_2)]$ be the submodule generated by $z_1 - \varphi(z_2)$. The submodule $M_\varphi$ was studied by the second and the third author in \cite{IY08}. Let $\{\lambda_k(z_2)\}$ be an orthonormal basis of $K_\varphi(z_2) = H^2(z_2) \ominus \varphi H^2(z_2)$,
$$e_j = \frac{z_2^j + z_2^{j-1}z_1 + \cdots + z_1^j}{\sqrt{j+1}}, j \geq 0$$
and $E_{k,j} = \lambda_k(z_2) e_j(z_1, \varphi(z_2))$. Let $S_{z_1}^\varphi = P_{M_\varphi^\perp}M_{z_1}|M_\varphi^\perp$ and define the operator
$$V: H^2(\D^2) \ominus M_\varphi \rightarrow K_\varphi(z_2) \otimes L^2_a(\D)$$
by $V(E_{k,j}) = \lambda_k(z_2) \sqrt{j+1}z^j$. It is shown in \cite{IY08} that $\{E_{k,j}\}$ is an orthonormal basis of $H^2(\D^2) \ominus M_\varphi$, $V$ is a unitary operator and
$$V S_{z_1}^\varphi = (I\otimes M_z) V,$$
i.e., $S_{z_1}^\varphi$ is unitarily equivalent to $I\otimes M_z$. It is clear that $I\otimes M_z$ is a Fredholm operator on $K_\varphi(z_2) \otimes L^2_a(\D)$ if and only if $K_\varphi(z_2)$ is finite dimensional, or equivalently, if and only if
 $\varphi$ is a finite Blaschke product. Now we take a look at a submodule $\HM$ which contains $z_1 - \varphi(z_2)$ (but not necessarily generated by it) and study its Hilbert-Schmidtness under the assumption that $\varphi$ is a finite Blaschke product. Observe that in this case there exists a closed subspace $\HM_1 \subseteq H^2(\D^2) \ominus M_\varphi$ such that $\HM = \HM_1 \oplus M_\varphi$. We extend $V$ to be zero on $M_\varphi$ and denote the new operator also by $V$, then $V^*: K_\varphi(z_2) \otimes L^2_a(\D) \rightarrow H^2(\D^2)$ is an isometry with range $H^2(\D^2) \ominus M_\varphi$ and $V: H^2(\D^2) \rightarrow K_\varphi(z_2) \otimes L^2_a(\D)$ is a partial isometry. Let $\HN = V\HM = V \HM_1$. Then clearly $\HN$ is invariant under $I\otimes M_z$. Define $S_{z_i} = P_{\HM^\perp} M_{z_i} |_{\HM^\perp},  i = 1, 2$, and $S_\HN = P_{\HN^\perp} (I \otimes M_z) |_{\HN^\perp}$. We will see that $S_{z_1}$ is unitarily equivalent to $S_\HN$. Since it is well-known that submodules $\HM$ with $\dim \HM^\perp < \infty$ are Hilbert-Schmidt, we assume in the sequal that $\dim \HM^\perp = \infty$.
\begin{lemma}\label{clrfsnczp}
Let $\HM \in Lat(H^2(\D^2))$ contain $z_1 - \varphi(z_2)$ and $\HN = V\HM$, where $\varphi$ is a finite Blaschke product. Then for every $\lambda\in {\mathbb D}$ the operator $S_\HN-\lambda$ has closed range.
\end{lemma}
\begin{proof}
It is equivalent to show that $ran (S_\HN^*-\overline{\lambda})$ is closed. We only verify the case for $\lambda=0$ since the general case is similar. It is clear that $I\otimes M_z^*: K_\varphi(z_2) \otimes L^2_a(\D) \rightarrow K_\varphi(z_2) \otimes L^2_a(\D)$ has closed range and $\ker (I\otimes M_z^*) = K_\varphi(z_2)$. Observe that $S_\HN^* = (I\otimes M_z^*) |_{\HN^\perp}$. We then have
$$S_\HN^* \HN^\perp = (I\otimes M_z^*) (\HN^\perp + K_\varphi(z_2)) = (I\otimes M_z^*) [(\HN^\perp + K_\varphi(z_2)) \ominus K_\varphi(z_2)].$$
Since $\varphi$ is a finite Blaschke product, we have $\dim K_\varphi(z_2) < \infty$. Thus $\HN^\perp + K_\varphi(z_2)$ is closed, and so $S_\HN^* \HN^\perp$ is closed. The proof is complete.
\end{proof}
Note that $\ker S_\HN^* = K_\varphi(z_2) \cap \HN^\perp$, we conclude from the above lemma that $S_\HN$ is a semi-Fredholm operator.
\begin{lemma}\label{sFosnzcp}
Let $\HM \in Lat(H^2(\D^2))$ contain $z_1 - \varphi(z_2)$ and $\HN = V\HM$, where $\varphi$ is a finite Blaschke product. Then for every $\lambda \in \D$, the operator $S_\HN - \lambda$ is semi-Fredholm with \[\ind (S_\HN - \lambda) = \dim (\HN \ominus (I\otimes M_z)\HN) - \dim K_\varphi.\]
\end{lemma}
\begin{proof}
In view of Lemma 3.1 we only need to consider the index of $S_\HN - \lambda$. To this end we write
\begin{equation*}
I \otimes M_z = \left( \begin{array}{cc}
    (I \otimes M_z)|_{\HN} & A\\

    0 & S_\HN\\

  \end{array}\right)
\end{equation*}
with respect to the decomposition $K_\varphi(z_2) \otimes L^2_a(\D) = \HN \oplus \HN^\perp$. Then for $\lambda \in \D$, we have
\begin{align}\label{madnanp}
I \otimes M_z - \lambda& =
\begin{pmatrix}
  (I \otimes M_z)|_{\HN} - \lambda & A\\
    0 & S_\HN - \lambda\\
 \end{pmatrix}\\
&= \begin{pmatrix}
  I & 0\\
    0 & S_\HN - \lambda\\
 \end{pmatrix}
 \begin{pmatrix}
  I & A\\
    0 & I\\
 \end{pmatrix}
 \begin{pmatrix}
  (I \otimes M_z)|_{\HN} - \lambda & 0\\
    0 & I\\
 \end{pmatrix}.\notag
 \end{align}
 It is clear that $ \begin{pmatrix}
  I & A\\
    0 & I\\
 \end{pmatrix}$ is invertible. Since the Fredholm index of a product equals the sum of the indices, we obtain
\begin{align*}
-\dim K_\varphi &= \ind (I \otimes M_z - \lambda) = \ind (S_\HN - \lambda) + \ind ((I \otimes M_z)|_{\HN} - \lambda).
\end{align*}
Since $(I \otimes M_z)|_{\HN} - \lambda$ is known to be semi-Fredholm for every $\lambda\in {\mathbb D}$ and ${\mathbb D}$ is path connected, we have
\[ \ind ((I \otimes M_z)|_{\HN} - \lambda)=\ind ((I \otimes M_z)|_{\HN})= \dim (\HN \ominus (I\otimes M_z)\HN).\]
Thus we have
\[\ind (S_\HN - \lambda) = \dim (\HN \ominus (I\otimes M_z)\HN) - \dim K_\varphi,\]
when all the numbers involved are finite.
 Furthermore, if $\dim (\HN \ominus (I\otimes M_z)\HN) = \infty$, then $(I \otimes M_z)|_{\HN} - \lambda$ is not a Fredholm operator, so in the Calkin algebra its image is not invertible. Hence (\ref{madnanp}) implies $S_\HN - \lambda$ is not a Fredholm operator, i.e., $\ind (S_\HN - \lambda) = \infty$. The proof is complete.
\end{proof}
Recall that $S_{z_i} = P_{\HM^\perp} M_{z_i} |\HM^\perp, i = 1, 2$. Now we determine the essential spectrum for $S_{z_1}$.
\begin{lemma}\label{espfszozp}
Let $\HM \in Lat(H^2(\D^2))$ contain $z_1 - \varphi(z_2)$ and $\HN = V\HM$, where $\varphi$ is a finite Blaschke product.\\
(i) If $\dim (\HN \ominus (I\otimes M_z)\HN) = \infty$, then $\sigma_e(S_{z_1}) = \overline{\D}$.\\
(ii) If $\dim (\HN \ominus (I\otimes M_z)\HN) < \infty$, then $\sigma_e(S_{z_1}) \subseteq \T$.
\end{lemma}
\begin{proof}
Recall that $V^*: K_\varphi(z_2) \otimes L^2_a(\D) \rightarrow H^2(\D^2)$ is an isometry with range $H^2(\D^2) \ominus M_\varphi$. So $VV^* = I$ and $V^*V = P_{\HM_\varphi^\perp}$. Since $\HM = \HM_1 \oplus M_\varphi$ for some $\HM_1 \subseteq H^2(\D^2) \ominus M_\varphi$, and $V^* \HN = \HM_1$, we conclude that $V^* (\HN^\perp) = \HM^\perp$. Recall also that $V S_{z_1}^\varphi = (I\otimes M_z) V$, it then follows that
\begin{align*}
S_{z_1}^\varphi V^* &= P_{\HM_\varphi^\perp}M_{z_1}P_{\HM_\varphi^\perp} V^*\\
&= V^*(V S_{z_1}^\varphi)V^* = V^*[(I\otimes M_z) V]V^*\\
& = V^*(I\otimes M_z).
\end{align*}
Thus
$$S_{z_1}V^*|\HN^\perp = V^*|\HN^\perp P_{\HN^\perp} (I\otimes M_z) |\HN^\perp = V^*|\HN^\perp S_\HN.$$
So $S_{z_1}$ is unitarily equivalent to $S_\HN$. The assertions then follow from Lemma \ref{sFosnzcp}.
\end{proof}

We need the following theorem from \cite{Ya01} to study the Hilbert-Schmidtness of a submodule.
\begin{theorem}[\cite{Ya01}]\label{grshcm}
Let $\HM \subseteq H^2(\D^2)$ be a submodule. If $\D$ is not a subset of $\sigma_e(S_{z_1}) \cap \sigma_e(S_{z_2})$, then $\HM$ is a Hilbert-Schmidt submodule.
\end{theorem}
The following result is immediate.
\begin{corollary}\label{mrischszp}
Let $\HM \in Lat(H^2(\D^2))$ contain $z_1 - \varphi(z_2)$ and $\HN = V\HM$, where $\varphi$ is a finite Blaschke product. If $\dim (\HN \ominus (I\otimes M_z)\HN) < \infty$, then $\HM$ is a Hilbert-Schmidt submodule.
\end{corollary}
\begin{proof}
If $\dim (\HN \ominus (I\otimes M_z)\HN) < \infty$, then Lemma \ref{espfszozp} (ii) ensures that $\sigma_e(S_{z_1}) \subseteq \T$. Thus by Theorem \ref{grshcm}, we conclude that $\HM$ is a Hilbert-Schmidt submodule.
\end{proof}
Before we prove Theorem \ref{hlsmnmcdv}, we need some lemmas.
\begin{lemma}\label{dmineqlm}
Let $\HM \in Lat(H^2(\D^2))$ contain $z_1 - \varphi(z_2)$, where $\varphi$ is an inner function. Then
$$\dim (\HM \ominus (z_1\HM + \varphi(z_2)\HM)) \leq \dim (H^2(\D^2) \ominus [z_1, \varphi(z_2)]) \rank \HM.$$
\end{lemma}
\begin{proof}
If $\dim (H^2(\D^2) \ominus [z_1, \varphi(z_2)])$ or $\rank \HM$ is infinity, then there is nothing to prove. So suppose $\dim (H^2(\D^2) \ominus [z_1, \varphi(z_2)])$ and $\rank \HM$ are finite. Suppose $\{e_i\}$ is an orthonormal basis of $H^2(\D^2) \ominus [z_1, \varphi(z_2)]$, and $\HM = [f_1, f_2, \cdots, f_{n+1}]$, where $f_j \in H^2(\D^2), f_{n+1} = z_1 - \varphi(z_2), j = 1, \cdots, n$. Let $P_\varphi$ be the orthogonal projection onto $\HM \ominus (z_1\HM + \varphi(z_2)\HM)$. We claim that $\HM \ominus (z_1\HM + \varphi(z_2)\HM)$ is contained in $\text{span}\{P_{\varphi} (e_if_j), i, j \geq 1\}$. Then the conclusion will follow from this claim.

Now we prove the claim. Suppose $g \in \HM \ominus (z_1\HM + \varphi(z_2)\HM)$ and $g$ is orthogonal to $\text{span}\{P_{\varphi} (e_if_j), i, j \geq 1\}$. Then for any polynomial $h$, there are $h_1 \in H^2(\D^2) \ominus [z_1, \varphi(z_2)], h_2 \in [z_1, \varphi(z_2)]$ such that $h = h_1 + h_2$. Note that $P_\varphi(h_1 f_j)$ is in $\text{span}\{P_{\varphi} (e_if_j), i, j \geq 1\}$ and $h_2f_j$ is in the closure of $z_1\HM + \varphi(z_2)\HM$. Thus
$$\langle g, hf_j\rangle = \langle g, h_1 f_j\rangle + \langle g, h_2f_j\rangle = 0.$$
Since $\HM$ is generated by $\{f_j\}$, we conclude that $g = 0$. So the claim holds and the proof is complete.
\end{proof}

\begin{prop}\label{indeve}
Let $\HM \in Lat(H^2(\D^2))$. If $[R_{z_1}^*,R_{z_2}]$ is compact, then $\forall \lambda, \eta \in \D^2$, $F_\lambda$ and $F_\eta$ are left semi-Fredholm operators and $\ind F_\lambda = \ind F_\eta$.
\end{prop}
\begin{proof}
If $[R_{z_1}^*,R_{z_2}]$ is compact, then $[R_{\varphi_{\lambda_1}}^*,R_{\varphi_{\lambda_2}}]$ is compact for any $\lambda \in \D^2$. Hence Propositions \ref{rlspoaci} and \ref{srsfgl} ensure that $F_\lambda$ and $G_\eta$ are left semi-Fredholm operators. Since $F_\lambda = F_{(\lambda_1,0)} - \lambda_2, G_\eta = G_{(0,\eta_2)} - \eta_1$, it follows that
$$\ind F_\lambda = \ind F_{(\lambda_1, \eta_2)}, \quad \ind G_\eta = \ind G_{(\lambda_1, \eta_2)}.$$
Note that $F_\lambda$ and $G_\lambda$ have the same cokernel, and Lemma \ref{rlbfadots} implies that $\ker F_\lambda$ and $\ker G_\lambda$ are isomorphic. Therefore the conclusion follows from the above two equations.
\end{proof}

\begin{lemma}\label{hsipfgcp}
Let $\HM \in Lat(H^2(\D^2))$ contain $z_1 - \varphi(z_2)$, where $\varphi$ is a finite Blaschke product. If $\HM$ is a Hilbert-Schmidt submodule, then the space $\HM \ominus (z_1\HM + \varphi(z_2)\HM)$ is of finite dimensional.
\end{lemma}
\begin{proof}
If $\HM$ is a Hilbert-Schmidt submodule, then $[R_{z_1}^*,R_{z_1}][R_{z_2}^*,R_{z_2}]$ and $[R_{z_1}^*,R_{z_2}]$ are Hilbert-Schmidt operators. It then follows from Propositions \ref{rlspoaci} and \ref{indeve} that $(z_1 - \lambda_1)\HM + (z_2 - \lambda_2)\HM$ is closed and $\ind_\lambda\HM = \dim (\HM \ominus (z_1 - \lambda_1)\HM + (z_2 - \lambda_2)\HM) < \infty$. Without loss of generality, suppose $\varphi(0) = 0$, i.e. $\varphi(z_2) = z_2 \psi(z_2)$, where $\psi(z_2)$ is a finite Blaschke product. Note that by induction, we only need to prove the case when $\psi(z_2)$ is a m\"{o}bius transform. So suppose $\psi(z_2) = \frac{\alpha - z_2}{1-\overline{\alpha}z_2}: = \phi_\alpha(z_2)$. Now we show $\dim(\HM \ominus (z_1\HM + z_2\phi_\alpha(z_2)\HM)) < \infty$.

Notice that $\dim(\HM \ominus (z_1\HM + \phi_\alpha(z_2)\HM) < \infty$. Define
$$T: \HM \Big/ (z_1\HM + \phi_\alpha(z_2)\HM) \rightarrow (z_1\HM + z_2\HM)\Big/(z_1\HM + z_2\phi_\alpha(z_2)\HM)$$
by $T([g])=[z_2 g], g \in \HM$. By a verification, we see that $T$ is well defined and $T$ is onto. Thus $\dim((z_1\HM + z_2\HM)/(z_1\HM + z_2\phi_\alpha(z_2)\HM)) < \infty$. Since
$$\HM \ominus (z_1\HM + z_2\phi_\alpha(z_2)\HM) = \HM \ominus (z_1\HM + z_2\HM) \oplus [(z_1\HM + z_2\HM) \ominus (z_1\HM + z_2\phi_\alpha(z_2)\HM)],$$
we obtain that $\dim(\HM \ominus (z_1\HM + z_2\phi_\alpha(z_2)\HM)) < \infty$. The proof is complete.
\end{proof}

Now we can prove Theorem \ref{hlsmnmcdv}.
\begin{proof}[Proof of Theorem \ref{hlsmnmcdv}]
Suppose $\HM$ is finitely generated. Since $H^2(\D^2) \ominus [z_1, \varphi(z_2)] = K_\varphi(z_2)$, Lemma \ref{dmineqlm} asserts that $\dim (\HM \ominus (z_1\HM + \varphi(z_2)\HM))$ is finite. Let $\HN = V\HM$. It is not difficult to check that $V^*(\HN \ominus (I\otimes M_z)\HN) \subseteq \HM \ominus (z_1\HM + \varphi(z_2)\HM)$. Thus $\dim(\HN \ominus (I\otimes M_z)\HN) < \infty$. Hence by Corollary \ref{mrischszp}, the submodule $\HM$ is Hilbert-Schmidt.

For the necessity, if $\HM$ is a Hilbert-Schmidt submodule, then Lemma \ref{hsipfgcp} implies $\dim(\HM \ominus (z_1\HM + \varphi(z_2)\HM))< \infty$. Thus $\dim(\HN \ominus (I\otimes M_z)\HN) < \infty$. Note that $K_\varphi(z_2) \otimes L^2_a(\D)$ is isomprphic to $\C^k \otimes L^2_a(\D)$, where $k$ is the order of $\varphi$. By Theorem 3.6 in \cite{Sh01}, we have $\HN = [\HN \ominus (I\otimes M_z)\HN]$. Therefore $\HN$ is finitely generated. One verifies that $\HM = [V^*(\HN \ominus (I\otimes M_z)\HN), z_1-\varphi(z_2)]$. So $\HM$ is finitely generated.
\end{proof}

\begin{corollary}\label{nchsfgszp}
Let $\HM = [f_1, \cdots, f_n, z_1 - \varphi(z_2)]$, where $\varphi$ is a finite Blaschke product and $f_j \in H^2(\D^2), j = 1, \cdots, n$, are arbitrary. Then $\HM$ is a Hilbert-Schmidt submodule.
\end{corollary}

\subsection{Submodules containing $z_1 - z_2$} In this subsection, we consider the special case $\varphi(z_2) = z_2$ and fully characterize the submodules containing $z_1 - z_2$.
 In this case, since the space $K_\varphi(z_2) \otimes L^2_a(\D) = L^2_a(\D)$, we can write out the operators $V$ and $V^*$ more explicitly. Indeed,
$$V: H^2(\D^2) \rightarrow L^2_a(\D)$$
is the operator defined by $Vf(\lambda) = f(\lambda, \lambda)$, and \[V^*g(z_1,z_2) = \frac{1}{z_2 - z_1} \int_{z_1}^{z_2} g(\lambda) d\lambda.\] One checks that $\ker V = [z_1 - z_2]$ and $V^*$ is an isometry. Suppose $\HN \in Lat(M_z, L^2_a(\D))$, let $\HM = \tau(\HN): = V^* \HN + \ker V$, then $\HM \in Lat(H^2(\D^2))$. Note that $\tau$ defines a one-to-one correspondence between $Lat(M_z, L^2_a(\D))$ and submodules in $Lat(H^2(\D^2))$ that contain $\ker V$ (\cite{Ri87}). Let $\HM_0 = [z_1 - z_2]$. Then $P_{\HM_0^\perp}M_{z_1}|_{\HM_0^\perp} = P_{\HM_0^\perp}M_{z_2}|_{\HM_0^\perp}$, and $P_{\HM_0^\perp}M_{z_1}|_{\HM_0^\perp}$ is unitarily equivalent to the Bergman shift $M_z$ on the Bergman space $L^2_a(\D)$. In fact, $P_{\HM_0^\perp}M_{z_1}|_{\HM_0^\perp} V^* = V^* M_z$ on $L^2_a(\D)$ (see also \cite{DP89, GSZZ09}). The following lemma is proved in \cite{LR}.
\begin{lemma}[\cite{LR}]\label{rbhanthbs}
Let $\HM \in Lat(H^2(\D^2))$ contain $z_1 - z_2$, and $\HN = V\HM$. Then for every $\lambda \in \D$, we have $\HM = [V^*(\HN \ominus (z-\lambda)\HN), z_1 - z_2]$ and $\ind \HN \leq \ind_{(\lambda,\lambda)}\HM \leq \ind \HN + 1$, where $\ind \HN = \dim (\HN \ominus z\HN)$.
\end{lemma}
In fact, suppose $\HM \in Lat(H^2(\D^2))$ contains $z_1 - z_2$ and $\HN = V\HM$, by $\HN = [\HN \ominus (z-\lambda)\HN]$ (\cite{ARS96}), we have $\HM = [V^*(\HN \ominus (z-\lambda)\HN), z_1 - z_2], \lambda \in \D$. Note that for $f, g \in \HM, h \in \HN \ominus (z-\lambda)\HN$,
$$\langle V^*h, (z_1 - \lambda)f + (z_2 - \lambda)g\rangle_{H^2(\D^2)} = \langle h, (z-\lambda)V(f+g)\rangle_{L^2_a} = 0.$$
So $V^*(\HN \ominus (z-\lambda)\HN) \subseteq \HM \ominus ((z_1 - \lambda)\HM + (z_2 - \lambda)\HM)$. Since $\ind_{(\lambda,\lambda)}\HM$ is less than or equal to the rank of $\HM$ and $\dim (\HN \ominus (z-\lambda)\HN) = \dim (\HN \ominus z\HN) = \ind \HN$, we immediately obtain $\ind \HN \leq \ind_{(\lambda,\lambda)}\HM \leq \ind \HN + 1$. Thus if $\HM \in Lat(H^2(\D^2))$ contains $z_1 - z_2$, then $\HM$ is finitely generated if and only if $\HN$ is finitely generated, which is equivalent to the condition that $\ind_{(0,0)}\HM < \infty$. By Lemmas \ref{espfszozp} and \ref{rbhanthbs}, we obtain the following result.

\begin{lemma}\label{untebsamr}
Let $\HM \in Lat(H^2(\D^2))$ contain $z_1 - z_2$, and $\HN = V\HM$.\\
(i) If $\ind_{(0,0)}\HM = \infty$, then $\sigma_e(S_{z_i}) = \overline{\D}, i = 1, 2$.\\
(ii) If $\ind_{(0,0)}\HM < \infty$, then $\sigma_e(S_{z_i}) \subseteq \T, i = 1, 2$.
\end{lemma}

We need the following theorem from \cite{GRS05} to prove Theorem \ref{hlsmnmcdv2}. For $f \in H^2(\D^2)$, we write $Z(f) = \{\lambda \in \D^2: f(\lambda) = 0\}$ and $Z(\HM) = \bigcap_{f\in \HM} Z(f)$.
\begin{theorem}[\cite{GRS05}]\label{GRS05np}
If a submodule $\HM$ of $H^2(\D^2)$ contains a nonzero bounded function $\varphi$, then
$$\sigma_e(R) \cap \D^2 \subseteq Z(\varphi)$$
and for every $\lambda \in \D^2 \setminus \sigma_e(R)$ the pair $R - \lambda$ has Fredholm index 1. In fact, for all $\lambda \in \D^2 \setminus Z(\varphi)$ we have
$$\dim \HM /((z_1-\lambda_1)\HM + (z_2-\lambda_2)\HM) = 1.$$
\end{theorem}

Now we can prove the following theorem.
\begin{theorem}\label{hlsmnmcdv2}
Let $\HM \in Lat(H^2(\D^2))$ contain $z_1 - z_2$. The following are equivalent.\\
(i) $\ind_{(0,0)}\HM < \infty$.\\
(ii) $\HM$ is a Hilbert-Schmidt submodule.\\
(iii) $[R_{z_1}^*,R_{z_2}]$ is compact.\\
(iv) $F_\lambda$ is a semi-Fredholm operator for some $\lambda = (\lambda_0, \lambda_0) \in \D^2$.
\end{theorem}
\begin{proof}
(i) implies (ii). If $\ind_{(0,0)}\HM < \infty$, then Lemma \ref{untebsamr} ensures that $\sigma_e(S_{z_i}) \subseteq \T$. It then follows from Theorem \ref{grshcm} that $\HM$ is a Hilbert-Schmidt submodule.\\
(ii) implies (iii). This follows from definition.\\
(iii) implies (iv). If $[R_{z_1}^*, R_{z_2}]$ is compact, then Proposition \ref{indeve} asserts that $F_\lambda$ is a semi-Fredholm operator for all $\lambda \in \D^2$.\\
(iv) implies (i). Suppose $F_\lambda$ is semi-Fredholm for some $\lambda = (\lambda_0, \lambda_0)$. Note that Theorem \ref{GRS05np} and (\ref{rlbindofr}) imply that for $\lambda = (\lambda_1,\lambda_2) \in \D^2$ with $\lambda_1 \neq \lambda_2$, $1 = \ind (R - \lambda) = -\ind F_\lambda$. It thus follows that $F_{(\lambda_0, \lambda_0)}$ is Fredholm. So $\dim \ker F_{(\lambda_0, \lambda_0)}^* = \ind_{(\lambda_0, \lambda_0)} \HM < \infty$. Then Lemma \ref{rbhanthbs} ensures that $\ind_{\lambda}\HM < \infty$, $\forall \lambda \in \D^2$ with $\lambda_1 = \lambda_2$. In particular, $\ind_{(0,0)}\HM < \infty$. The proof is complete.
\end{proof}
The equivalence of (i) and (iv) in the above theorem generalizes Theorem 2.9 in \cite{III}.
\begin{corollary}\label{hlsmmcdvc}
Let $\HM \in Lat(H^2(\D^2))$ contain $z_1 - z_2$. Then $\HM$ is a Hilbert-Schmidt submodule if and only if $\sigma_e(S_{z_1})  \neq \overline{\D}$.
\end{corollary}
\begin{proof}
If $\sigma_e(S_{z_1})  \neq \overline{\D}$, then by Theorem \ref{grshcm}, we conclude that $\HM$ is a Hilbert-Schmidt submodule. Conversely, if $\HM$ is a Hilbert-Schmidt submodule, then by Theorem \ref{hlsmnmcdv} and Lemma \ref{untebsamr}, we get the assertion.
\end{proof}

The following result characterizes the Fredholmness of the pairs $R - \lambda$ and $S - \lambda$ for $\lambda \in \D^2$.
\begin{prop}\label{fhnotrs}
Let $\HM \in Lat(H^2(\D^2))$ contain $z_1 - z_2$. Then the following are equivalent.\\
(i) $\ind_{(0,0)}\HM < \infty$.\\
(ii) $\forall \lambda \in \D^2$ the pair $R - \lambda$ is Fredholm with index 1.\\
(iii) $\forall \lambda \in \D^2$ the pair $S - \lambda$ is Fredholm with index 0.
\end{prop}
\begin{proof}
By Lemma \ref{isosps} and Theorem \ref{hlsmnmcdv2}, we see that (ii) implying (i) and (iii) implying (i) hold. It is left to show that (i) implies (ii) and (iii). If $\ind_{(0,0)}\HM < \infty$, then Theorem \ref{hlsmnmcdv2} ensures that $\HM$ is a Hilbert-Schmidt submodule. Thus $R - \lambda$ and $S - \lambda$ are Fredholm pairs for $\lambda \in \D^2$ (\cite{Ya01, Ya06}). Since for $\lambda \in \D^2$ with $\lambda_1 \neq \lambda_2$, we have $R - \lambda$ and $S - \lambda$ are Fredholm with index $1$ and $0$, respectively (\cite{GRS05}). The assertion follows from this.
\end{proof}
Let $\HM \in Lat(H^2(\D^2))$ contain $z_1 - z_2$. It is proved in \cite{LR} that $\ran \partial_{R-\lambda}^1$ is closed for $\lambda \in \D^2$. It is also proved in \cite[Lemma 2.6]{III} that $\ran \partial_{R}^1 = z_1 \HM + z_2 \HM$ is closed. We use a similar argument as in \cite[Lemma 2.6]{III} to prove the closedness of $\ran \partial_{R-\lambda}^1$ in the following. Note that this result holds even when $\ind_{(0,0)}\HM = \infty$.
\begin{lemma}\label{clsfdrd}
Let $\HM \in Lat(H^2(\D^2))$ contain $z_1 - z_2$ and $\HN = V \HM$. Then for $\lambda = (\lambda_0,\lambda_0) \in \D^2$, $\ran \partial_{R-\lambda}^1 = (z_1 - \lambda_0)\HM + (z_2 - \lambda_0)\HM$ is closed.
\end{lemma}
\begin{proof}
Since $[z_1 - z_2]$ is generated by the polynomial $z_1 - z_2$, we have $[z_1 - z_2]$ is a Hilbert-Schmidt submodule (\cite{Ya99}). So for $\lambda = (\lambda_0,\lambda_0) \in \D^2$, $(z_1 - \lambda_0)[z_1 - z_2] + (z_2 - \lambda_0)[z_1 - z_2]$ is closed and
\begin{align}\label{hsicark}
[z_1 - z_2] = (z_1 - \lambda_0)[z_1 - z_2] + (z_2 - \lambda_0)[z_1 - z_2] + \C (z_1 - z_2).
\end{align}
Note that $\HM = V^*\HN \oplus [z_1 - z_2]$. Let $L_0 = (z_1 - \lambda_0)[z_1 - z_2] + (z_2 - \lambda_0)[z_1 - z_2]$, we then have
\begin{align}\label{idfdrlz}
(z_1 - \lambda_0)\HM + (z_2 - \lambda_0)\HM = (z_1 - \lambda_0)V^*\HN + (z_2 - \lambda_0)V^*\HN + L_0.
\end{align}
Notice that
\begin{align*}
&V\left((z_1 - \lambda_0)\HM + (z_2 - \lambda_0)\HM\right) = (z-\lambda_0)\HN = V(V^*(z-\lambda_0)\HN)\\
& = V(V^*(z-\lambda_0)\HN \oplus L_0) = V(V^*(z-\lambda_0)\HN \oplus [z_1 - z_2]).
\end{align*}
It follows from (\ref{hsicark}) and (\ref{idfdrlz}) that
\begin{align}\label{cotibsa}
V^*(z-\lambda_0)\HN \oplus L_0 \subseteq (z_1 - \lambda_0)\HM + (z_2 - \lambda_0)\HM \subseteq V^*(z-\lambda_0)\HN \oplus [z_1 - z_2].
\end{align}
It is known that $(z-\lambda_0)\HN$ is closed, thus $V^*(z-\lambda_0)\HN$ is closed, so $V^*(z-\lambda_0)\HN \oplus L_0$ and $V^*(z-\lambda_0)\HN \oplus [z_1 - z_2]$ are closed. Since
$$V^*(z-\lambda_0)\HN \oplus [z_1 - z_2] = V^*(z-\lambda_0)\HN \oplus L_0 + \C(z_1 - z_2),$$
we conclude from (\ref{cotibsa}) that $(z_1 - \lambda_0)\HM + (z_2 - \lambda_0)\HM$ is closed.
\end{proof}

Similar result holds for the pair $S - \lambda$.
\begin{lemma}\label{clotosf}
Let $\HM \in Lat(H^2(\D^2))$ contain $z_1 - z_2$. Then for $\lambda = (\lambda_0,\lambda_0) \in \D^2$, $ran \partial_{S-\lambda}^0$ and $\ran \partial_{S-\lambda}^1$ are closed.
\end{lemma}

\begin{proof}
We prove the lemma for $\lambda = (0,0)$, the other cases follow by a similar argument. First we show $\ran \partial_S^0$ is closed. It is equivalent to show $\ran \partial_S^{0*}$ is closed. Let $\HN = V\HM$, then $\HM = V^*\HN + \ker V$ and $\HM^\perp =  V^*(\HN^\perp)$. Note that $\ran \partial_S^{0*} = M_{z_1}^* \HM^\perp + M_{z_2}^* \HM^\perp$. So
\begin{align*}
\ran \partial_S^{0*} &= M_{z_1}^* V^*(\HN^\perp) + M_{z_2}^* V^*(\HN^\perp)\\
& = V^* M_z^* (\HN^\perp) + V^* M_z^* (\HN^\perp)\\
& = V^* M_z^* (\HN^\perp).
\end{align*}
Therefore $\ran \partial_S^{0*}$ is closed.

Next we show $\ran \partial_S^1$ is closed. It is equivalent to show that $\ran \partial_S^{1*}$ is closed. Note that
$$\partial_S^{1*} = \left(
\begin{matrix}
-M_{z_2}^*|_{\HM^\perp}\\
M_{z_1}^*|_{\HM^\perp}
\end{matrix}
\right)
:\HM^\perp \rightarrow \left(
\begin{matrix}
\HM^\perp\\
\HM^\perp
\end{matrix}.
\right)$$ Let $\Lambda^* = \left(
\begin{matrix}
-M_{z_2}^*\\
M_{z_1}^*
\end{matrix}
\right)$. Since $\Lambda^*: H^2(\D^2)\rightarrow \left(
\begin{matrix}
H^2(\D^2)\\
H^2(\D^2)
\end{matrix}
\right)$ has closed range and $\ker \Lambda^* = \C$, applying the same reasoning as in Lemma \ref{clrfsnczp}, we see that $\ran \partial_S^{1*}$ is closed.
\end{proof}

The following two lemmas are needed to study the dimensions for the cohomology spaces for the pairs $R - \lambda$ and $S - \lambda$.
\begin{lemma}\label{indfnnlz}
Let $\HN \in Lat(M_z,L^2_a(\D))$.\\
(i) If $\lambda_0 \in Z(\HN)$, let $\HN_0 = \HN /\varphi_0$, where $\varphi_0(z) = \frac{z-\lambda_0}{1-\overline{\lambda_0}z}$, then $\ind \HN_0 = \ind \HN$.\\
(ii) If $\lambda_0 \not\in Z(\HN)$, let $\HN_1 = \{f \in \HN: f(\lambda_0) = 0\}$, then $\ind (\HN_1/\varphi_0) = \ind \HN$.
\end{lemma}
\begin{proof}
(i) Let $U_0$ be the operator on $L^2_a(\D)$ defined by $U_0 f (z)= f(-\varphi_0(z)) \frac{1-|\lambda_0|^2}{(1-\overline{\lambda_0}z)^2}$, then $U_0$ is a unitary operator. Note that $U_0 \HN_0 = (U_0 \HN)/z$, so $\ind \HN_0 = \ind ((U_0 \HN)/z)$. By \cite[Lemma 2.1]{III} or \cite{Ja83} or \cite[Proposition 3]{Zhu98}, we have $\ind ((U_0 \HN)/z) = \ind (U_0 \HN)$. Thus $\ind \HN_0 = \ind (U_0 \HN) = \ind \HN$.

(ii) Since $U_0 (\HN_1/\varphi_0) = (U_0 \HN_1)/z$, it follows that $\ind (\HN_1/\varphi_0) = \ind ((U_0 \HN_1)/z)$. Notice that $U_0 \HN_1 = \{g \in U_0 \HN: g(0) = 0\}$, so $(U_0 \HN_1)/z = \{h \in L^2_a(\D): zh \in U_0 \HN\}$. Then \cite[Proposition 5]{Zhu98} implies that $\ind ((U_0 \HN_1)/z) = \ind U_0 \HN$. Hence $\ind (\HN_1/\varphi_0) = \ind \HN$.
\end{proof}

\begin{lemma}\label{dimfmsfm}
Let $\HM \in Lat(H^2(\D^2))$ contain $z_1 - z_2$ and $\HN = V\HM$. Then for $\lambda = (\lambda_0,\lambda_0) \in \D^2$,
$$ \dim (\ker \partial_{S-\lambda}^{1} \ominus\ran \partial_{S-\lambda}^{0}) =
\begin{cases}
\ind \HN + 1, \lambda \in Z(\HM),\\
\ind \HN - 1, \lambda \not\in Z(\HM).
\end{cases}
$$
\end{lemma}

\begin{proof}
Suppose $\lambda = (\lambda_0,\lambda_0) \in \D^2$. Note that
$$\ker \partial_{S-\lambda}^1 = \{(f,g): q(z_1-\lambda_1)g = q(z_2-\lambda_2) f, f, g \in \HM^\perp\},$$
$$\ran \partial_{S-\lambda}^0 = \{(q(z_1-\lambda_1) f, q(z_2-\lambda_2) f): f \in \HM^\perp\}.$$
Let
$$\Lambda^1 = \{(f,g): S_{\varphi_{\lambda_1}}g = S_{\varphi_{\lambda_2}} f, f, g \in \HM^\perp\},$$
$$\Lambda^0 = \{(S_{\varphi_{\lambda_1}} f, S_{\varphi_{\lambda_2}} f): f \in \HM^\perp\}.$$
We define $W: \ker \partial_{S-\lambda}^1 \rightarrow \Lambda^1$ by
$$W(f,g) = (q(1-\overline{\lambda_2}z_2 )f, q(1-\overline{\lambda_1}z_1) g).$$
It is not difficult to verify that $W$ is one-to-one and onto, and $W (\ker \partial_{S-\lambda}^1 / \ran \partial_{S-\lambda}^0) = \Lambda^1 / \Lambda^0$. Thus $\dim (\ker \partial_{S-\lambda}^1 / \ran \partial_{S-\lambda}^0) = \dim \Lambda^1 / \Lambda^0$. Notice that
$$\Lambda^1 \ominus \Lambda^0 = \{(f,g): S_{\varphi_{\lambda_1}}g = S_{\varphi_{\lambda_2}} f, M_{\varphi_{\lambda_1}}^* f + M_{\varphi_{\lambda_2}}^* g = 0, f, g \in \HM^\perp\}.$$
Let $\varphi_0(z) = \frac{z-\lambda_0}{1-\overline{\lambda_0}z}, S_\HN^* = M_{\varphi_{0}}^*|\HN^\perp$ on $\HN^\perp$ and set \[I = \{(f_1, g_1): S_\HN g_1 = S_\HN f_1, M_{\varphi_0}^* (f_1 + g_1) = 0, f_1, g_1 \in \HN^\perp\}.\] We define the map
$$T: \Lambda^1 \ominus \Lambda^0 \rightarrow I$$
by sending $(f,g)$ to $(Vf, Vg)$. Then $T$ is one-to-one and onto. Thus $\dim (\Lambda^1 \ominus \Lambda^0) = \dim I$. Now we determine $\dim I$.

(i) If $\lambda_0 \in Z(\HN)$, then $(\frac{1}{(1-\overline{\lambda_0}z)^2},\frac{1}{(1-\overline{\lambda_0}z)^2}) \in I$. Let $\HN_0 = \HN / \varphi_0$ and $(f_1, g_1) \in I$, then $\varphi_0g_1 - \varphi_0f_1 =h_1$ for some $h_1 \in \HN$. Thus $g_1 - f_1 = h_1 / \varphi_0 \in \HN^\perp \cap \HN_0$. Now define
$$A: I \rightarrow \HN^\perp\cap\HN_0$$
by $A(f_1,g_1) = g_1 - f_1$. If $A(f_1,g_1) = g_1 - f_1 = 0$, then from $M_{\varphi_0}^* (f_1 + g_1) = 0$, we have $f_1 = g_1 = c\frac{1}{(1-\overline{\lambda_0}z)^2}, c \in \C$. On the other hand, for $h_1/\varphi_0 \in \HN^\perp \cap \HN_0$, let $f_1 = \frac{-h_1}{2\varphi_0}$ and $g_1 = \frac{h_1}{2\varphi_0}$. Then $(f_1,g_1) \in I$ and $A (f_1,g_1) = h_1/\varphi_0$. Hence $A$ is onto. Therefore $\dim I = 1 + \dim \HN^\perp \cap \HN_0$. Note that $\HN_0 = \left(\HN_0 \ominus \varphi_0\HN_0\right) \oplus \HN$, so $\dim I = 1+ \ind \HN_0$. It then follows from Lemma \ref{indfnnlz} that $\dim I = 1 + \ind \HN$.

(ii) If $\lambda_0 \not\in Z(\HN)$, let $Q_\HN$ be the projection onto $\HN$ and $\HN_1 = \{h\in \HN: h(\lambda_0) = 0\}$, then $\HN_1 = \HN \ominus \C Q_\HN \frac{1}{(1-\overline{\lambda_0}z)^2}$. Let $(f_1, g_1) \in I$, then $\varphi_0g_1 - \varphi_0f_1 =h_1$ for some $h_1 \in \HN$. Hence $h_1 \in \HN_1$ and $h_1 / \varphi_0 \in \HN_1/\varphi_0 \cap \HN^\perp$. Similarly, we define
$$X: I \rightarrow \HN^\perp\cap\HN_1/\varphi_0$$
by $X(f_1,g_1) = g_1 - f_1$. Then one checks that $X$ is one-to-one and onto. So $\dim I = \dim (\HN^\perp \cap \HN_1/\varphi_0)$. Notice that $\HN_1/\varphi_0 = (\HN_1/\varphi_0 \ominus \HN) \oplus (\HN \ominus \HN_1) \oplus \HN_1$, thus \[\dim I = \dim (\HN^\perp \cap \HN_1/\varphi_0) = \dim (\HN_1/\varphi_0 \ominus \HN_1) - 1.\] Lemma \ref{indfnnlz} then ensures that $\dim I = \dim (\HN_1/\varphi_0 \ominus \HN_1) - 1 = \ind \HN - 1$. The proof is complete.
\end{proof}
Now we determine the dimensions for the cohomology vector spaces for the pairs $R - \lambda$ and $S - \lambda$.
\begin{prop}\label{dmfass}
Let $\HM \in Lat(H^2(\D^2))$ contain $z_1 - z_2$. Then for $\lambda \in \D^2$\\
(i) $$\dim\ker F_\lambda = \dim (\ker \partial_{R - \lambda}^1 \ominus \ran \partial_{R - \lambda}^0) = \dim \ker \partial_{S-\lambda}^{0} = \ind_\lambda \HM - 1;$$
(ii) $$ \dim (\ker \partial_{S-\lambda}^{1} \ominus\ran \partial_{S-\lambda}^{0}) =
\begin{cases}
\ind_\lambda \HM, \lambda \in Z(\HM),\\
\ind_\lambda \HM - 1, \lambda \not\in Z(\HM);
\end{cases}
$$
(iii) $$ \dim(\ran \partial_{S-\lambda}^{1})^\perp =
\begin{cases}
1, \lambda \in Z(\HM),\\
0, \lambda \not\in Z(\HM).
\end{cases}
$$
\end{prop}
\begin{proof}
Note that $(\ran \partial_{S-\lambda}^{1})^\perp = \C \frac{1}{(1-\overline{\lambda_1}z_1)(1-\overline{\lambda_2}z_2)} \cap \HM^\perp$, thus (iii) is true. Recall from \cite{GRS05} that for $\lambda \in \D^2$ with $\lambda_1 \neq \lambda_2$, $R - \lambda$ and $S - \lambda$ are Fredholm with index $1$ and $0$, respectively. Recall also that $\ind (R- \lambda) = \ind_\lambda \HM - \dim (\ker \partial_{R - \lambda}^1 \ominus \ran \partial_{R - \lambda}^0)$. Therefore Lemma \ref{isosps} implies (i) is true for $\lambda \in \D^2$ with $\lambda_1 \neq \lambda_2$. Since $\ind (S - \lambda) = 0$ for $\lambda \in \D^2$ with $\lambda_1 \neq \lambda_2$, we conclude that (ii) also holds for $\lambda \in \D^2$ with $\lambda_1 \neq \lambda_2$. Now we consider $\lambda\in \D^2$ with $\lambda_1 = \lambda_2$. Suppose $\lambda = (\lambda_0,\lambda_0) \in \D^2$, we have two cases.

If $\ind_\lambda \HM < \infty$, then $\ind_{(0,0)}\HM < \infty$. Thus Proposition \ref{fhnotrs} asserts that $R - \lambda$ and $S - \lambda$ are Fredholm with index $1$ and $0$. So the same argument as above implies (i) and (ii) hold in this case.

If $\ind_\lambda \HM = \infty$, then $\ind \HN = \infty$. Hence by Lemma \ref{dimfmsfm}, we get $\dim (\ker \partial_{S-\lambda}^{1} \ominus\ran \partial_{S-\lambda}^{0}) = \infty$. So (ii) is true. Next we show that $\dim\ker F_\lambda = \infty$. Suppose $\dim\ker F_\lambda < \infty$. Lemma \ref{clsfdrd} assures that $\ran \partial_{R-\lambda}^1$ is closed. Thus $\ran F_\lambda$ is closed and $F_\lambda$ is semi-Fredholm. Then Theorem \ref{hlsmnmcdv2} shows $\ind_{(0,0)}\HM < \infty$, and so $\ind_\lambda \HM < \infty$. This is a contradiction. So (i) holds and the proof is complete.
\end{proof}

The following corollary is an immediate consequence of the above proposition.

\begin{corollary}\label{eqftqiny}
Let $\HM \in Lat(H^2(\D^2))$ contain $z_1 - z_2$. Then the following are equivalent.\\
(i) $\ind_{(0,0)}\HM = \infty$.\\
(ii) $\dim \partial_{S-\lambda}^0 = \infty, \forall \lambda = (\lambda_0,\lambda_0)\in \D^2$.\\
(iii) $\dim (\ker \partial_{S-\lambda}^{1} \ominus\ran \partial_{S-\lambda}^{0}) = \infty, \forall \lambda = (\lambda_0,\lambda_0)\in \D^2$.
\end{corollary}

Before ending the paper, let us take another look at Theorem 1.1. Let
$\varphi(z_2)=\prod_{j=1}^n\frac{z_2-\lambda_j}{1-\overline{\lambda_j}z_2}$ be a finite Blaschke product. Since $|\lambda_j|<1$ for each $j$, the product  $q(z_2)=\prod_{j=1}^n(1-\overline{\lambda_j}z_2)$ is a polynomial such that $|q(z_2)|\geq \prod_{j=1}^n(1-|\lambda_j|)>0$ on ${\mathbb D}$. Hence a submodule $\HM$ contains $z_1-\varphi(z_2)$ if and only if it contains the polynomial $z_1q(z_2)-\prod_{j=1}^n(z_2-\lambda_j)$. The next conjecture is thus a natural weakening of that in \cite{Ya99}.\\

{\bf Conjecture}. Let $\HM$ be a submodule that contains a nontrivial polynomial. Then $\HM$ is Hilbert-Schmidt if and only if it is finitely generated.


\end{document}